\documentclass[10pt]{amsart}
\usepackage{a4wide}

\author{Manuel Bodirsky}
    \address{Laboratoire d'Informatique  (LIX), CNRS UMR 7161\\
    Ecole Polytechnique \\91128 Palaiseau\\
    France}
    \email{bodirsky@lix.polytechnique.fr}
    \urladdr{http://www.lix.polytechnique.fr/~bodirsky/}
\author{Michael Pinsker}
    \address{Laboratoire de Math\'{e}matiques Nicolas Oresme\\
    CNRS UMR 6139\\ Universit\'{e} de Caen\\14032 Caen Cedex\\
    France}
    \email{marula@gmx.at}
    \urladdr{http://dmg.tuwien.ac.at/pinsker/}
    \thanks{The second author is grateful for support through Erwin Schr\"{o}dinger Fellowship N2742-N18 of the Austrian Science Fund.}

\title{All reducts of the random graph are model-complete}


\usepackage{cite}
\usepackage{amsmath, amsxtra, amsfonts, amssymb}
\usepackage{mathrsfs,amsthm,url}
\theoremstyle{plain}
    \newtheorem{thm}{Theorem}
    \newtheorem{theorem}[thm]{Theorem}
    \newtheorem{lem}[thm]{Lemma}
    
    \newtheorem{proposition}[thm]{Proposition}
    \newtheorem{cor}[thm]{Corollary}

\theoremstyle{definition}
    \newtheorem{defn}[thm]{Definition}
    \newtheorem{definition}[thm]{Definition}
    
\theoremstyle{remark}

\newcommand{\sw}{{\it sw}}
\newcommand{\nin}{\notin}
\providecommand{\Aut}{\mathop{\rm Aut}\nolimits}
\providecommand{\End}{\mathop{\rm End}\nolimits}

\newcommand{\Q}{{\mathscr Q}}
\newcommand{\U}{{\mathscr U}}
\newcommand{\C}{{\mathscr C}}
\newcommand{\R}{{\mathscr R}}
\newcommand{\F}{{\mathscr F}}

\newcommand{\A}{{\mathscr A}}
\newcommand{\B}{{\mathscr B}}
\renewcommand{\P}{{\mathscr P}}

\newcommand{\G}{{\mathscr G}}
\newcommand{\M}{{\mathscr M}}

\newcommand{\T}{{\mathscr T}}

\newcommand{\K}{{\mathscr K}}
\renewcommand{\H}{{\mathscr H}}

\renewcommand{\S}{{\mathscr S}}
\newcommand{\To}{\rightarrow}
\newcommand{\cal}[1]{\ensuremath{\mathcal{#1}}}

\begin{document}

%-------------------------------------------------------------------------
\begin{abstract}
    We study locally closed transformation monoids which contain the
    automorphism group of the random graph. We show that such a
    transformation monoid is locally generated by the permutations in
    the monoid, or contains a constant operation, or contains an
    operation that maps the random graph injectively to an induced
    subgraph which is a clique or an independent set.

    As a corollary, our techniques yield a new proof of Simon Thomas'
    classification of the five closed supergroups of the automorphism
    group of the random graph; our proof uses different Ramsey-theoretic
    tools than the one given by Thomas, and is perhaps more
    straightforward.

    Since the monoids under consideration are endomorphism monoids of
    relational structures definable in the random graph, we are able to
    draw several model-theoretic corollaries: One consequence of our
    result is that all structures with a first-order definition in the
    random graph are model-complete. Moreover, we obtain a
    classification of these structures up to existential
    interdefinability.
\end{abstract}
\maketitle

\section{Introduction}\label{sect:prelims}
The random graph (also called the Rado graph) is the graph $G=(V;E)$
defined uniquely up to isomorphism by the property that for all
finite disjoint subsets $U,U'$ of the countably infinite vertex set
$V$ there exists a vertex $v \in V \setminus (U \cup U')$ such that
$v$ is in $G$ adjacent to all vertices in $U$ and to no vertex in
$U'$; we will refer to this property of the random graph as the
\emph{extension property}. For the many remarkable properties of
this graph and its automorphism group, and various connections to
many branches of mathematics, see
e.g.~\cite{RandomCameron,RandomRevisitedCameron}.

Simon Thomas has classified the five locally closed supergroups of
the automorphism group of $G$ in \cite{RandomReducts}. In this paper
we more generally investigate locally closed transformation monoids
that contain the automorphism group of $G$. We show that every such
monoid is either a disguised group in the sense that it is generated
by the largest permutation group which it contains, or it contains a
constant operation, or an injective operation which either deletes
all edges or all non-edges of the random graph. As a by-product of
our proof, we obtain a new proof of Thomas' classification.

Not surprisingly, insights on the behavior of functions on $G$ have
consequences for model-theoretic questions concerning the random
graph. Every closed supergroup of the automorphism group of $G$ is
the automorphism group of a relational structure definable in $G$;
such structures are called \emph{reducts} of $G$. Moreover, two
reducts $\Gamma_1,\Gamma_2$ have the same automorphism group iff
they are first-order interdefinable, i.e., iff every relation of
$\Gamma_1$ has a first-order definition in $\Gamma_2$, and
vice-versa. Thus, Thomas' theorem is the classification of the
reducts of $G$ up to first-order interdefinability. By considering
monoids of self-embeddings instead of automorphism groups, we obtain
a finer classification of these reducts, namely up to
\emph{existential} interdefinability, i.e., we do not distinguish
between two structures $\Gamma_1, \Gamma_2$ whenever every relation
of $\Gamma_1$ is definable in $\Gamma_2$ by an existential
first-order formula, and vice versa.

Another consequence of our results is that all reducts $\Gamma$ of
the random graph are \emph{model-complete}, i.e., all embeddings
between models of the first-order theory of $\Gamma$ preserve all
first-order formulas. The analogous statement for the reducts of
$(\mathbb Q;<)$, the dense linear order of the rationals, follows
from \cite[Proposition 8]{tcsps}. Model-completeness is a central
concept in model theory; see e.g.~\cite{Hodges}. For example,
model-completeness plays an important role when establishing
quantifier-elimination results. Whether or not a structure is
model-complete is usually not preserved by first-order
interdefinability. In this light, the result that all reducts of the
random graph are model-complete might be surprising.

The results presented are also relevant for the study of the
\emph{constraint satisfaction problem} for structures with a
first-order definition in the random graph. When $\Gamma$ is a
structure with a finite relational signature $\tau$, then the
constraint satisfaction problem for $\Gamma$ (denoted by
CSP$(\Gamma)$) is the computational problem of deciding whether a
given primitive positive sentence over $\tau$ is true in $\Gamma$. A
formula is called \emph{primitive positive} iff it is of the form
$\exists x_1,\dots,x_n. \psi_1 \wedge \dots \wedge \psi_m$ where
$\psi_1,\dots,\psi_m$ are atomic. The complexity of CSP$(\Gamma)$
does not change when $\Gamma$ is expanded by finitely many relations
with a primitive positive definition in $\Gamma$. Even though
expansions by relations with an existential positive definition
might increase the complexity of the constraint satisfaction
problem,  the classification of the reducts of $(\mathbb Q;<)$ up to
existential positive interdefinability was an important ingredient
in a recent complexity classification for the CSP of such
reducts~\cite{tcsps}. The results in this paper pave the way for a
similar classification for reducts of the random graph.

\section{Results}

We now present our main results, formulated in terms of
transformation monoids and permutation groups; the proofs of these
results will be of purely combinatorial nature. The model-theoretic
corollaries of the results presented  here will be drawn in
Section~\ref{sect:result_modeltheory}.

A monoid $\M$ of mappings from a set $D$ to $D$ is called
\emph{(locally) closed} iff the following holds: whenever $f: D
\rightarrow D$ is such that for every finite $A \subseteq D$ there
exists $e\in\M$ such that $e(x)=f(x)$ for all $x \in A$, then $f$ is
 an element of $\M$. Equivalently, the monoid is a closed set in
the product topology of $D^D$, where $D$ is taken to be discrete.
For the purposes of this paper, we call the smallest closed
transformation monoid that contains a set of operations $F$ from $V$
to $V$ and the automorphism group $\Aut(G)$ of the random graph the
 \emph{monoid generated by $F$}.

Similarly, a permutation group $\G$ acting on $D$ is called
(locally) closed iff it is closed in the subspace of $D^D$
consisting of all permutations on $D$; equivalently, $\G$ contains
all permutations which can be interpolated by elements of $\G$ on
arbitrary finite subsets of $D$, as in the definition of a closed
monoid above. As before, we call the smallest closed group
containing a set of permutations $F$ on $V$ as well as $\Aut(G)$ the
\emph{group generated by $F$}.

The random graph contains all countable graphs as induced subgraphs.
In particular, it contains an infinite complete subgraph, denoted by
$K_\omega$. It follows from the homogeneity of $G$ (see
Section~\ref{sect:result_modeltheory}) that all injective operations
from $V$ to $V$ whose image induces $K_\omega$ in $G$ locally
generate the same monoid. Let $e_E$ be one such injective operation
whose image induces $K_\omega$ in $G$. Similarly, $G$ contains an
infinite independent set, denoted by $I_\omega$. Let $e_N$ be an
injective operation from $V \rightarrow V$ whose image induces
$I_\omega$ in $G$.

Our main result is the following. It states that all closed monoids
containing $\Aut(G)$ either contain a quite primitive function, or
are generated by their permutations. As it turns out, the
permutations in such a monoid form a closed group.

\begin{theorem}\label{thm:endos}
For any closed monoid $\M$ containing $\Aut(G)$, one of the
following cases applies.
\begin{enumerate}
\item $\M$ contains a constant operation.
\item $\M$ contains $e_E$.
\item $\M$ contains $e_N$.
\item $\M$ is generated by (the closed group of) its permutations.
\end{enumerate}
\end{theorem}

The last case splits into five sub-cases, corresponding to the five
locally closed permutation groups that contain $\Aut(G)$. These
groups have already been exhibited by Thomas~\cite{RandomReducts}.
In our proof of Theorem~\ref{thm:endos}, we will be forced to
re-derive this result. While our proof of that classification, being
the proof of a more general result, is longer than the one in
\cite{RandomReducts}, it might be more canonical (in the sense of
Definition~\ref{def:canonical}). We now define the five groups.

It is clear that the complement graph of $G$ is isomorphic to $G$.
Note that by the homogeneity of $G$ any isomorphism between $G$ and
its complement locally generates the same transformation monoid
(group). Let $-$ be one such isomorphism.

For any finite subset $S$ of $V$, if we flip edges and non-edges
between $S$ and $V \setminus S$ in $G$, then the resulting graph is
isomorphic to $G$ (it is straightforward to verify the extension
property). Let $i_S$ be such an isomorphism for each non-empty
finite $S$. Every such operation generates the same transformation
monoid (group). We also write \sw\ for $i_{\{0\}}$, where $0 \in V$
is a fixed element for the rest of the paper, and refer to this
operation as the \emph{switch}.

\begin{theorem}[of~\cite{RandomReducts}]\label{thm:reducts}
Let $\G$ be a closed permutation group containing $\Aut(G)$. Then
exactly one out of the following five cases is true.
\begin{enumerate}
\item $\G$ equals $\Aut(G)$.
\item $\G$ is the group generated by $-$.
\item $\G$ is the group generated by $\sw$.
\item $\G$ is the group generated by $\{-,\sw\}$.
\item $\G$ is the group of all
permutations on $V$.
\end{enumerate}
\end{theorem}

The arguments given in~\cite{RandomReducts} use a Ramsey-theoretic
result by Ne\v{s}et\v{r}il~\cite{NesetrilNocliques}, namely that the
class of all finite graphs excluding finite cliques of a fixed size
forms a \emph{Ramsey class} (in the sense of~\cite{NesetrilSurvey}).
We also use a Ramsey-theoretic result, shown by R\"odl and
Ne\v{s}et\v{r}il~\cite{NesetrilRoedlPartite,NesetrilRoedlOrderedStructures}
(and independently by~\cite{AbramsonHarrington}), which is
different: we need the fact that finite ordered vertex-colored
graphs form a Ramsey class. We believe that our approach is
canonical, and that the proof techniques could very well be adapted
to show similar classifications for supergroups of automorphism
groups of other infinite structures $\Gamma$ which have the property
that the class of all finite structures that embed into $\Gamma$
(possibly equipped with a linear order on the vertices) is a Ramsey
class.

\section{Model-theoretic corollaries}\label{sect:result_modeltheory}

We now discuss the results of the preceding section in a
model-theoretic setting and establish some corollaries in this
language.

One easily verifies that the endomorphism monoid $\End(\Delta)$
(automorphism group $\Aut(\Delta)$) of a structure $\Delta$ with
domain $D$ is a closed monoid (group) on $D$, and that every closed
monoid (group) is of this form for an adequate structure $\Delta$
(confer also \cite[Corollary~1.9]{Szendrei} for these concepts).
Moreover, the automorphism group of a \emph{reduct} $\Gamma$ of a
structure $\Delta$, i.e., of a structure $\Gamma$ which is
first-order definable in $\Delta$, clearly contains $\Aut(\Delta)$.
The following is Theorem~\ref{thm:endos}, restated in terms of
structures.

\begin{theorem}
Let $\Gamma$ be first-order definable in the random graph. Then one
of the following cases applies.
\begin{enumerate}
\item $\Gamma$ has a constant endomorphism.
\item $\Gamma$ has the endomorphism $e_E$.
\item $\Gamma$ has the endomorphism $e_N$.
\item $\End(\Gamma)$ is generated by $\Aut(\Gamma)$.
\end{enumerate}
\end{theorem}

For automorphism groups of reducts of the random graph, we have even
more. It is well-known that the random graph is \emph{(ultra-)
homogeneous}, i.e., every isomorphism between two finite induced
substructures of $G$ can be extended to an automorphism of $G$ (see
\cite[Theorem 6.4.4]{Hodges}). For relational structures with a
finite signature, homogeneity implies \emph{$\omega$-categoricity}:
all countable models of the first-order theory of $G$ are isomorphic
(Corollary~6.4.2 of \cite{Hodges}). Reducts of $\omega$-categorical
structures are $\omega$-categorical (see e.g.
\cite[Theorem~6.3.6]{Hodges}).

 Now, the theorem of Engeler,
Ryll-Nardzewski, and Svenonius (see
e.g.~\cite[Theorem~6.3.1]{Hodges}) states that a relation $R$ is
first-order definable in an $\omega$-categorical structure $\Delta$
 if and only if $R$ is preserved by all automorphisms of $\Delta$. As a consequence, the reducts of an
$\omega$-categorical structure $\Delta$ are, up to first-order
interdefinability, in one-to-one correspondence with the locally
closed permutation groups containing $\Aut(\Delta)$. To illustrate
this, we restate Theorem~\ref{thm:reducts} by means of this
connection.

On the random graph, let $R^{(k)}$ be the $k$-ary relation that
holds on $x_1,\dots,x_k \in V$ if $x_1,\dots,x_k$ are pairwise
distinct, and the number of edges between these $k$ vertices is odd.
Note that $R^{(4)}$ is preserved by $-$, $R^{(3)}$ is preserved by
$\sw$, and that $R^{(5)}$ is preserved by $-$ and by $\sw$, but not
by all permutations of $V$.

\begin{theorem}[of~\cite{RandomReducts}]\label{thm:reducts2}
Let $\Gamma$ be a structure with a first-order definition in the
random graph $(V;E)$. Then exactly one out of the following five
cases is true.
\begin{enumerate}
\item $\Gamma$ is first-order interdefinable with $(V;E)$.
\item $\Gamma$ is first-order interdefinable with $(V;R^{(4)})$.
\item $\Gamma$ is first-order interdefinable with $(V;R^{(3)})$.
\item $\Gamma$ is first-order interdefinable with $(V;R^{(3)},R^{(4)})$.
\item $\Gamma$ is first-order interdefinable with $(V;=)$.
\end{enumerate}
\end{theorem}

For any reduct $\Gamma$, a case of Theorem~\ref{thm:reducts2}
applies iff the case with the same number applies for $\Aut(\Gamma)$
in Theorem~\ref{thm:reducts}. We will not prove this relational
description in this paper; however, given Theorem~\ref{thm:reducts}
and the discussion above, verifying the equivalence is merely an
exercise.

In the same way as automorphisms can be used to characterize
first-order definability, self-embeddings can be used to
characterize existential definability, and endomorphisms can be used
to characterize existential positive definability in
$\omega$-categorical structures. This is the content of the
following theorem. We say that a first-order formula is
\emph{existential} iff it is of the form $\exists x_1,...,x_n.
\psi$, where $\psi$ is quantifier-free, and \emph{existential
positive} iff it is existential and positive, i.e., in addition it
does not contain any negations.

\begin{theorem}\label{thm:pres}
A relation $R$ has an existential positive (existential) definition
in an $\omega$-categorical structure $\Gamma$ if and only if $R$ is
preserved by the endomorphisms (self-embeddings) of $\Gamma$.
\end{theorem}
\begin{proof}
It is easy to verify that existential positive formulas are
preserved by endomorphisms, and existential formulas are preserved
by self-embeddings of $\Gamma$.

For the other direction, note that the endomorphisms and
self-embeddings of $\Gamma$ contain the automorphisms of $\Gamma$,
and hence the theorem of Ryll-Nardzewski shows that $R$ has a
first-order definition in $\Gamma$; let $\phi$ be a formula defining
$R$. Suppose for contradiction that $R$ is preserved by all
endomorphisms of $\Gamma$ but has no existential positive definition
in $\Gamma$. We use the homomorphism preservation theorem (see
\cite[Section 5.5, Exercise~2]{Hodges}),
 which states that a
first-order formula $\phi$ is equivalent to an existential positive
formula modulo a first-order theory $T$ if and only if $\phi$ is
preserved by all homomorphisms between models of $T$. Since by
assumption $\phi$ is not equivalent to an existential positive
formula in $\Gamma$, there are models $\Gamma_1$ and $\Gamma_2$ of
the first-order theory of $\Gamma$ and a homomorphism $h$ from
$\Gamma_1$ to $\Gamma_2$ that violates $\phi$. By the Theorem of
L\"owenheim-Skolem (see e.g.~\cite{Hodges}) the first-order theory
of the two-sorted structure $(\Gamma_1,\Gamma_2;h)$ has a countable
model $(\Gamma_1',\Gamma_2';h')$. Since both $\Gamma_1'$ and
$\Gamma_2'$ must be countably infinite, and because $\Gamma$ is
$\omega$-categorical, we have that $\Gamma_1'$ and $\Gamma_2'$ are
isomorphic to $\Gamma$, and $h'$ can be seen as an endomorphism of
$\Gamma$ that violates $\phi$; a contradiction.

The argument for existential definitions and self-embeddings is
similar, but instead of the homomorphism preservation theorem we use
the Theorem of {\L}os-Tarski which states that a first-order formula
$\phi$ is equivalent to an existential formula modulo a first-order
theory $T$ if and only if $\phi$ is preserved by all embeddings
between models of $T$ (see e.g.~\cite[Corollary 5.4.5]{Hodges}).
\end{proof}

The following proposition, which links the operational generating
process with preservation of relations of structures, is
easy to prove; see e.g.~\cite{Szendrei}.\\

\begin{proposition}~\label{prop:loc-gen} Let $F, H$ be sets of
mappings from $V$ to $V$. Then the monoid generated by $F$ contains
$H$ iff every relation first-order definable in $G$ and preserved by
$F$ is also preserved by $H$.
\end{proposition}

Using Theorem~\ref{thm:pres}, we obtain an interesting and perhaps
surprising consequence of our main result. A theory $T$ is called
\emph{model-complete} iff every embedding between models of $T$ is
elementary, i.e., preserves all first-order formulas. It is
well-known that a theory $T$ is model-complete if and only if every
first-order formula is modulo $T$ equivalent to an existential
formula (see~\cite[Theorem 7.3.1]{Hodges}). A structure is said to
be model-complete iff its first-order theory is model-complete. From
the definition of model-completeness and $\omega$-categoricity it is
easy to see that an $\omega$-categorical structure $\Gamma$ is
model-complete iff all embeddings of $\Gamma$ into itself preserve
all first-order formulas. It follows from a result
in~\cite[Proposition 8]{tcsps} (based on a proof of a result by
Cameron~\cite{Cameron5} from~\cite{JunkerZiegler}) that all reducts
of the linear order of the rationals $(\mathbb Q;<)$ are
model-complete. We now see that the same is true for the random
graph.

\begin{cor}\label{cor:mc}
Every structure $\Gamma$ with a first-order definition in the random
graph is model-complete.
\end{cor}
\begin{proof}
An $\omega$-categorical structure $\Gamma$ is model-complete if and
only if all embeddings of $\Gamma$ into itself are locally generated
by the automorphisms of $\Gamma$. To see this, first assume that the
automorphisms of $\Gamma$ locally generate the self-embeddings of
$\Gamma$, and let $\phi$ be a first-order formula. By the equivalent
characterization of model-completeness mentioned above it suffices
to show that $\phi$ is equivalent to an existential formula. Since
$\phi$ is preserved by automorphisms of $\Gamma$, it is by
Proposition~\ref{prop:loc-gen} also preserved by self-embeddings of
$\Gamma$. Then Theorem~\ref{thm:pres} implies that $\phi$ is
equivalent to an existential formula. Conversely, suppose that all
first-order formulas are equivalent to an existential formula in
$\Gamma$. Since existential formulas are preserved by
self-embeddings of $\Gamma$, also the first-order formulas are
preserved by self-embeddings of $\Gamma$. By
Proposition~\ref{prop:loc-gen}, the self-embeddings are locally
generated by the automorphisms of $\Gamma$.

We thus show that the self-embeddings of $\Gamma$ are generated by
its automorphisms. Note that when we expand $\Gamma$ by $\neq$ and
by $\neg R$ for every relation in $\Gamma$, then the resulting
structure has the same set of self-embeddings. Hence, we assume in
the following that $\Gamma$ contains $\neq$ and $\neg R$ for all
relations $R$, and hence that all endomorphisms of $\Gamma$ are
embeddings. We apply Theorem~\ref{thm:endos}. If Case (4) of the
theorem holds, we are done. Note that $\Gamma$ cannot have a
constant endomorphism since $\Gamma$ contains $\neq$. So suppose
that $\Gamma$ is preserved by $e_N$. The substructure of $\Gamma$
induced by the image $e_N[V]$ of $e_N$ has a first-order definition
in $(e_N[V];=)$ (since all occurrences of $E$ in a formula defining
a relation of $\Gamma$ can be replaced by $false$); but then, since
$e_N$ is an embedding, $\Gamma$ has a first-order definition in
$(V;=)$. The set of self-embeddings of $\Gamma$ is then the set of
all injective mappings from $V$ to $V$. It follows that $\Gamma$ is
model-complete, because the monoid of injective mappings from $V$ to
$V$ is locally generated by the permutations of $V$ (i.e., the
automorphisms of $\Gamma$). The argument for $e_E$ is analogous.
\end{proof}

In $\omega$-categorical structures, homogeneity is equivalent to
having \emph{quantifier-elimination} (Theorem~2.22 in \cite{Oligo}):
every first-order formula is in $G$ equivalent to a quantifier-free
first-order formula; hence, the random graph has quantifier
elimination. The same is not true for its reducts. For example, any
two 2-element substructures of the structure
$$\Gamma = (V; \{(x,y,z) \; | \; E(x,y) \wedge \neg E(y,z)\})$$
are isomorphic. But since there is a first-order definition of $G$
in $\Gamma$, an isomorphism between a 2-element substructure with an
edge and a 2-element substructure without an edge cannot be extended
to an automorphism of $\Gamma$. However, our results imply that a
structure $\Gamma$ with a first-order definition in the random graph
is homogeneous when $\Gamma$ is expanded by all relations with an
existential definition in $\Gamma$.

\begin{cor}\label{cor:exdefn}
Every structure $\Gamma$ with a first-order definition in the random
graph has quantifier-elimination if it is expanded by all relations
with an existential definition in $\Gamma$.
\end{cor}
\begin{proof}
 This follows directly from the model-completeness of $\Gamma$ and the fact mentioned above that in
model-complete structures first-order formulas are equivalent to
existential formulas (see e.g.~\cite[Theorem~7.3.1]{Hodges}).
\end{proof}

As another application of our main theorem, we refine
Theorem~\ref{thm:reducts2} by giving a finer (at least in theory)
classification of the reducts of the random graph.

\begin{cor}\label{cor:classificationUpToExistential}
Up to existential interdefinability, there are exactly five
different structures with a first-order definition in the random
graph.
\end{cor}
\begin{proof}
In the same way as in the proof of Corollary~\ref{cor:mc}, we can
use Theorem~\ref{thm:endos} to show that either the self-embeddings
of a reduct $\Gamma$ are generated by the automorphisms, and
$\Gamma$ is existentially interdefinable with one of the structures
described in Theorem~\ref{thm:reducts}; or otherwise $\Gamma$ has an
existential definition in $(V;=)$, which is again one of the five
cases from Theorem~\ref{thm:reducts}.
\end{proof}

The endomorphism monoid $\End(G)$ of the random graph has been
studied in~\cite{DelicDolinka,BonatoDelic,RandomGraphMonoid}. By
Theorem~\ref{thm:pres}, studying closed transformation monoids
containing $\End(G)$ is equivalent to studying structures with a
first-order definition in $G$ up to existential positive
interdefinability. A complete classification of all locally closed
transformation monoids that contain all permutations of $V$, and
hence of the reducts of $(V; =)$ up to existential positive
interdefinability, has been given in~\cite{BodChenPinsker}; there is
only a countable number of such monoids. The results of the present
paper are far from providing a full classification of the locally
closed transformation monoids that contain the automorphisms of the
random graph --- this is left for future investigation.

\section{Additional notions and notation}

We will write $E(x,y)$ or $(x,y)\in E$ to express that two vertices
$x,y\in V$ are adjacent in the random graph. The binary relation
$N(x,y)$ is defined by $\neg E(x,y)\wedge x\neq y$. Pairs $\{x,y\}$
with $N(x,y)$ are referred to as \emph{non-edges}.

We write function applications of $-$ without braces.

Often when we have a graph $\P=(P;D)$, and $S\subseteq P$, then for
notational simplicity we write $(S;D)$ for the subgraph of $\P$
induced by $S$, i.e., we ignore the fact that $D$ would have to be
restricted to $S^2$.

We say that an operation $e: V\To V$ (a set $F$ of operations from
$V$ to $V$) is \emph{generated} by a set of operations $H$ from $V$
to $V$ iff it is contained in the monoid generated by $H$.

\section{Ramsey-theoretic Preliminaries}\label{ssect:ramsey}

We prepare the proof of our main theorem by recalling some
Ramsey-type theorems and extending these theorems for our purposes.
The notions and results of this section are of an abstract
Ramsey-theoretic nature and do not refer to concrete structures such
as the random graph.

We start by recalling a theorem on ordered structures due to
 Ne\v set\v ril and R\"odl \cite{NesetrilRoedlOrderedStructures} which we will make heavy use of. Let
$\tau=\tau' \cup \{\prec\}$ be a relational signature, and let $\cal
C(\tau)$ be the class of all finite $\tau$-structures $\S$ where
$\prec$ denotes a linear order on the domain of $\S$. For
$\tau$-structures $\A, \B$, let ${\A}\choose{\B}$ be the set of all
substructures of $\A$ that are isomorphic to $\B$ (we also refer to
members of ${\A}\choose{\B}$ as \emph{copies of $\B$ in $\A$}). For
a finite number $k\geq 1$, a \emph{$k$-coloring} of the copies of
$\B$ in $\A$ is simply a mapping $\chi$ from ${{\A}\choose{\B}}$
into a
 set of size $k$.

\begin{definition}
For $\S,\H,\P \in \cal C(\tau)$ and $k \geq 1$, we write $\S
\rightarrow (\H)^\P_k$ iff for every $k$-coloring $\chi$ of the
copies of $\P$ in $\S$ there exists a copy $\H'$ of $\H$ in $\S$
such that all copies of $\P$ in $\H'$ have the same color (under
$\chi$).
\end{definition}

\begin{theorem}[of~\cite{NesetrilRoedlOrderedStructures,
NesetrilRoedlPartite, AbramsonHarrington}]\label{thm:ramsey} The
class $\cal C(\tau)$ of all finite relational ordered
$\tau$-structures is a \emph{Ramsey class}, i.e., for all $\H,\P \in
\cal C(\tau)$ and $k \geq 1$ there exists $\S \in \cal C(\tau)$ such
that $\S \rightarrow (\H)^\P_k$.
\end{theorem}

\begin{cor}
\label{cor:coloringEdgesAndNonEdgesOfG}
    For every finite graph $\H$ and for all colorings $\chi_E$ and $\chi_N$ of the edges and the non-edges of
    $G$, respectively, by finitely many colors, there exists an isomorphic copy of $\H$ in $G$ which both colorings are constant on.
\end{cor}
\begin{proof}
    Let $k$ be the number of colors used altogether by $\chi_E$ and $\chi_N$.
    Let $\prec$ be any total order on the domain of $\H$, and denote the structure obtained from $\H$ by adding the order $\prec$ to the signature
    by $\bar{\H}$. Consider the complete graph $\K_2$ on two vertices, and order its two vertices anyhow to arrive at a structure $\bar{\K_2}$. Then the
    coloring $\chi_E$ of the edges of $\H$ can be viewed as a
    coloring of the copies of $\bar{\K_2}$ in $\bar{\H}$. Let $\bar{\S}$ with $\bar{\S}\rightarrow (\bar{\H})^{\bar{\K_2}}_k$ be provided by
    the preceding theorem, and let $\S$ be $\bar{\S}$ without the order. Then $\S$ is a graph with the property that whenever
    we color its edges with $k$ colors, then there is a copy of $\H$
    in $\S$ all of whose edges have the same color. Now we repeat
    the argument for the non-edges, starting from $\S$ instead of
    $\H$. We then arrive at a graph $\T$ with the property that
    whenever we color its edges and non-edges by $k$ colors, then
    there is a copy $\H'$ of $\H$ in $\T$ such that all edges of $\H'$ have the same color,
    and such that
    non-edges of $\H'$ have the same color. $\T$ has a copy in $G$,
    proving the claim.
\end{proof}

We will not only need to color edges of graphs, but also of graphs
equipped with additional structure.

\begin{defn}
    An \emph{$n$-partitioned graph} is a structure $\U=(U;F,U_1,\ldots,U_n)$, where $(U;F)$ is a graph and each
    $U_i$ is a subset of $U$ such that the $U_i$ form a partition of $U$.
\end{defn}

\begin{defn}
    Let $\U=(U;F)$ be a graph, and let $S_1,S_2$ be disjoint subsets of $U$. Let
    $\chi$ be a coloring of the two-element
    subsets of $U$. We say
    that $\chi$ is \emph{canonical on $S_1$} iff the
    color of a two-element subset of $S_1$ depends only on
    whether this set is an
    edge or a non-edge. Similarly, we say that $\chi$ is
    \emph{canonical between $S_1$ and $S_2$} iff  the
    color of every pair $\{s_1,s_2\}$, where $s_1\in S_1$ and $s_2\in S_2$,
    depends only on whether or not this pair is an edge.
\end{defn}

\begin{defn}
    Let $\U=(U;F,U_1,\ldots,U_n)$ be an $n$-partitioned graph.
    We say that a coloring of the two-element subsets of
    $U$ is \emph{canonical on $\U$} iff it is canonical on all $U_i$ and between all distinct $U_i,U_j$.
\end{defn}

\begin{lem}[The $n$-partitioned graph Ramsey
lemma]\label{lem:partitionedGraphRamseyLemma}
    Let $n, k\geq 1$. For any finite $n$-partitioned graph
    $\U=(U;F,U_1,\ldots,U_n)$ there exists a finite $n$-partitioned graph
    $\Q=(Q;D,Q_1,\ldots,Q_n)$ with the property that for all colorings of
    the two-element subsets of $Q$ with $k$ colors,
    there exists a copy of $\U$ in $\Q$ on which the coloring is canonical.
\end{lem}
\begin{proof}

We show the lemma for $n=2$; the generalization to larger $n$ is
straightforward. For $n=2$, we apply Theorem~\ref{thm:ramsey} six
times: Once for the edges in $U_1$, once for the edges in $U_2$,
once for the edges between $U_1$ and $U_2$, and then the same for
all three kinds of non-edges.

In general, we would have to apply the theorem $2\ (n+{n
\choose{2}})$ times: Once for the edges of each part $U_i$, once for
the edges between any two distinct parts $U_i, U_j$, and then the
same for all non-edges on and between parts.

So assume $n=2$. We exhibit the idea in detail for the edges between
$U_1$ and $U_2$. Let $\prec$ be any total order on $U$ with the
property that $u_1\prec u_2$ for all $u_1\in U_1$, $u_2\in U_2$.
Consider the 2-partitioned graph ${\mathscr
L}^1=(\{a,b\};\{(a,b),(b,a)\},\{a\},\{b\})$ and order its vertices
by setting $a\prec b$; so ${\mathscr L}^1$ consists of two adjacent
vertices which are ordered somehow, and which lie in different
parts. By Theorem~\ref{thm:ramsey}, there exists an ordered
partitioned graph $\Q^1=(Q^1;D^1,Q_1^1,Q_2^1,\prec)$ such that $\Q^1
\rightarrow (\U)^{{\mathscr L}^1}_k$.

Now, if we change the order on $\Q^1$ in such a way that $r\prec s$
for all $r\in Q^1_1$ and all $s\in Q^1_2$ and such that the order
within the parts $Q^1_1, Q^1_2$ remains unaltered, then the
statement $\Q^1 \rightarrow (\U)^{{\mathscr L}^1}_k$ still holds:
For, given a coloring of the copies of ${\mathscr L}^1$ with respect
to the new ordering, we obtain a coloring of (possibly fewer) copies
of ${\mathscr L}^1$ with respect to the old ordering. There, we
obtain a copy $\U'$ of $\U$ such that all copies of ${\mathscr L}^1$
in $\U'$ have the same color. But in this copy, by the choice of the
order on $\U$, we have that $r\prec s$ for all $r\in U_1'$ and all
$s\in U_2'$. Therefore, this copy is also a substructure of $\Q^1$
with respect to the new ordering.

Since we can change the ordering on $\Q^1$ in the way described
above, the colorings of the copies of ${\mathscr L}^1$ are just
colorings of those pairs $\{r,s\}$, with $r\in Q_1^1$ and $s\in
Q_2^1$, which are edges.

Now we repeat the process with the structure ${\mathscr
L}^2=(\{a,b\};\{(a,b),(b,a)\},\{a,b\},\emptyset)$, ordered again by
setting $a\prec b$, starting with $\Q^1$. We then obtain a structure
$\Q^2$; this step takes care of the edges which lie within $U_1$.
After that we proceed with ${\mathscr
L}^3=(\{a,b\};\{(a,b),(b,a)\},\emptyset,\{a,b\})$, thereby taking
care of the edges within $U_2$. We then apply
Theorem~\ref{thm:ramsey} three more times with the structures
${\mathscr L}^4=(\{a,b\};\emptyset,\{a\},\{b\})$, ${\mathscr
L}^5=(\{a,b\};\emptyset,\{a,b\},\emptyset)$, and ${\mathscr
L}^6=(\{a,b\};\emptyset,\emptyset,\{a,b\})$, in order to ensure
homogeneous non-edges.
\end{proof}

The preceding lemma on partitioned graphs was an auxiliary tool to
cope with graphs which have some distinguished vertices, as defined
in the following.

\begin{defn}
    An $n$-constant graph is a structure $\U=(U;F,u_1,\ldots,u_n)$, where $\U=(U;F)$ is a graph, and $u_i\in U$ are distinct.
\end{defn}

Observe that $n$-constant graphs are no relational structures;
therefore, in order to apply Theorem~\ref{thm:ramsey}, we have to
make them relational: To every $n$-constant graph
$\U=(U;F,u_1,\ldots,u_n)$ we can assign an $n+2^n$-partitioned graph
$\tilde{\U}=(U;F,\{u_1\},\ldots,\{u_n\},U_1,\ldots,U_{2^n})$ in
which the $u_i$ belong to singleton sets, and in which for every
possible relative position (edge or non-edge) to the $u_i$ we have a
set $U_j$ of all elements in $U\setminus\{u_1,\ldots,u_n\}$ having
this position. (In the language of model theory, every of the
$n+2^n$ sets corresponds to a maximal quantifier-free $1$-type over
the structure $\U$.) We call the parts $U_i$ the \emph{proper} parts
of $\tilde{\U}$.

\begin{defn}
    Let $\U=(U;F,u_1,\ldots,u_n)$ be an $n$-constant graph. We say that a coloring of the two-element subsets of $U$ is \emph{canonical on $\U$} iff it is canonical on the corresponding $n+2^n$-partitioned graph.
\end{defn}

We now arrive at the goal of this section, namely the following
lemma, which we are going to apply to mappings on the random graph
numerous times in the sections to come.

\begin{lem}[The $n$-constant graph Ramsey lemma]\label{lem:constantGraphRamseyLemma}
    Let $n, k\geq 1$. For any finite $n$-constant graph
    $\U=(U;F,u_1,\ldots,u_n)$ there exists a finite $n$-constant graph
    $\Q=(Q;D,q_1,\ldots,q_n)$ with the property that for all colorings of the two-element subsets of
    $Q$ with $k$ colors, there exists a copy of $\U$ in
    $\Q$ on which the coloring is canonical.
\end{lem}
\begin{proof}
    Let $\tilde{\U}:=(U;F,\{u_1\},\ldots,\{u_n\}, U_1,\ldots,U_{2^n})$ be the
    partitioned graph associated with $\U$. We would
    like to use the partitioned graph Ramsey lemma (Lemma~\ref{lem:partitionedGraphRamseyLemma}) in order to
    obtain $\Q$; but we want the singleton sets $\{u_i\}$ of the
    partition to remain singletons, which is not guaranteed by that lemma.

    So consider the $2^n$-partitioned graph
    $\R:=(U\setminus\{u_1,\ldots,u_n\};F, U_1,\ldots,U_{2^n})$, and apply the partitioned
    graph Ramsey lemma to this graph to obtain a partitioned graph $\R^0$.

    Equip $\R^0$ with any linear order. Now consider the
    ordered $2^n$-partitioned graph ${\mathscr L}^1$ which has just
    one vertex, and whose first part contains this single vertex.
    Apply Theorem~\ref{thm:ramsey} in order to obtain an ordered partitioned graph $\R^1$ such that
    $\R^1\To (\R^0)^{{\mathscr L}^1}_{k^n}$.

    Next, consider the ordered $2^n$-partitioned graph ${\mathscr L}^2$ which has just
    one vertex, and whose second part contains this single vertex. Apply Theorem~\ref{thm:ramsey} in order to obtain
    an ordered
    partitioned graph $\R^2$ such that
    $\R^2\To (\R^1)^{{\mathscr L}^2}_{k^n}$.

    Repeat this procedure with the ordered
    $2^n$-partitioned graphs ${\mathscr L}^3,\ldots,{\mathscr L}^{2^n}$; ${\mathscr L}^i$ has its single vertex in its $i$-th part. We end up with an ordered partitioned graph
    $\R^{2^n}$. We now forget its order and denote the resulting structure by $\T=(T;C,T_1,\ldots,T_{2^n})$.

    $\T$ has the following property: Whenever we color its vertices with
    $k^n$ colors, then we find a copy of $\R^0$ in $\T$ such that the coloring is constant on each part of this copy.
    Hence, it has the property
    that if we color its two-element subsets and its vertices with
    $k$ and $k^n$ colors, respectively, then we find in it a copy of
    $\R$ on which the first coloring is canonical,
    and such that the color of the vertices depends only on the part the vertex lies in.

    Now consider the structure
    $\S:=(T\cup\{u_1,\ldots,u_n\};B,\{u_1\},\ldots,\{u_n\},T_1,\ldots,T_{2^n})$, where $B$ consists of the edges of $\T$, plus
    edges connecting the $u_i$ with the vertices of some parts $T_i$, depending on whether $u_i$ was in $\U$ connected to the vertices in $U_i$ or not.
    Clearly, $\S$ is the partitioned graph of the $n$-constant graph $\Q:=(T\cup\{u_1,\ldots,u_n\};B,u_1,\ldots,u_n)$.
    We claim that $\Q$ has the property we want to prove.
    Assume that we color the two-element subsets of $T\cup\{u_1,\ldots,u_n\}$ with $k$
    colors. We must find a copy of $\U$ in $\Q$ on which the coloring
    is canonical. Divide the coloring into two colorings, namely the coloring
    restricted to two-element subsets of $T$, and the coloring of two-element subsets which contain at least one element $u_i$
    outside $T$. The color of the sets $\{u_i,u_j\}$ completely outside $T$ is irrelevant for what we want to prove, so forget about these.

    Now the coloring of those sets which have exactly one element outside $T$ can be encoded in a coloring of the vertices of
    $T$: Each vertex is given one of $k^n$ colors, depending on the colors of its edges leading to
    $u_1,\ldots,u_n$. So we have encoded the original coloring into a
    coloring of two-elements subsets of $T$ and a coloring of the vertices of $T$.
    With our observation above, this proves the lemma.
\end{proof}

\section{Finding structure in mappings on the random graph}% of Theorem~\ref{thm:endos}}

In this section we show how to use the Ramsey-theoretic results from
the last section in our context. To warm up, we prove a simple
observation (Proposition~\ref{prop:interpolation}) applying
Corollary~\ref{cor:coloringEdgesAndNonEdgesOfG}. The proposition
states that any mapping on the random graph behaves quite simple on
arbitrarily large finite subgraphs.

\begin{defn}
    Let $e,f: V\To V$. We say that $e$ \emph{behaves as $f$ on $F \subseteq V$} iff there is
    an automorphism $\alpha$ of $G$ such that $f(x)=\alpha(e(x))$ for
    all $x \in F$. We say that \emph{$e$ interpolates $f$ modulo
    automorphisms} iff for every finite $F \subseteq V$ there is an
    automorphism $\beta$ of $G$ such that $e(\beta(x))$ behaves as $f$
    on $F$; so this is the case iff there exist automorphisms
    $\alpha,\beta$ such that $\alpha(e(\beta(x))=f(x)$ for all $x\in
    F$.
\end{defn}

Note that if $e$ interpolates $f$ modulo automorphisms, then it also
generates $f$. We now want to make precise what it means that
arbitrarily large structures have a certain property.

\begin{defn}
    Let $\tau$ be any signature and let $\cal C(\tau)$ be a class of finite $\tau$-structures closed under substructures and with the property that for any
    two structures in $\cal C(\tau)$ there exists a structure in $\C(\tau)$ containing both structures.
    We order $\cal C(\tau)$ by the embedding
    relation $\subseteq$. Let $P(w)$ be any property. We say that \emph{$P$ holds for arbitrarily large
    elements of $\cal C(\tau)$} iff for any $\F\in\cal C(\tau)$ there exists $\H\in\cal C(\tau)$  such that $\F\subseteq\H$ and $P(\H)$
    holds. We say that \emph{$P$ holds for all sufficiently large
    elements of $\cal C(\tau)$} iff there is an element $\F$
    of $\cal C(\tau)$ such that $P$ holds for $\H$ whenever $\F$ embeds
    into $\H$.
\end{defn}

Our properties $P(w)$ will be such that if $P(\H)$ holds, then $P$
also holds for all substructures of $\H$. The definition then says
that $P$ holds for arbitrarily large elements of $\cal C(\tau)$ iff
for any $\F\in\cal C(\tau)$ there is $\F'\in\cal C(\tau)$ isomorphic
to $\F$ such that $P(\F')$ holds.

Observe also that if arbitrarily large structures in $\cal C(\tau)$
have one of finitely many properties, then one property holds for
arbitrarily large elements of $\cal C(\tau)$.

\begin{proposition}\label{prop:interpolation}
Let $e:V \rightarrow V$ be a mapping on the random graph. Then $e$
interpolates either the identity, $e_E$, $e_N$, a constant function,
or $-$ modulo automorphisms.
\end{proposition}
\begin{proof}

We show that arbitrarily large finite subgraphs of $G$ have the
property that $e$ behaves on them like one of the operations of the
proposition. Since there are finitely many operations to choose
from, $e$ then behaves like one fixed operation $p$ from the list on
arbitrarily large finite subgraphs of the random graph. By the
homogeneity of the random graph, we can freely move finite graphs
around by automorphisms, proving that $e$ interpolates $p$.

So let $\F$ be any finite graph; we have to find a copy $\F'$ of
$\F$ in $G$ such that $e$ behaves like one of the mentioned
operations on this copy.

We color all pairs $\{x,y\}$ of distinct vertices of $G$
\begin{itemize}
\item by $1$ if $e(x)=e(y)$,
\item by $2$ if $E(e(x),e(y))$,
\item by $3$ if $N(e(x),e(y))$.
\end{itemize}

By Corollary~\ref{cor:coloringEdgesAndNonEdgesOfG} there exists a
copy $\F'$ of $\F$ in $G$ such that all edges and all non-edges of
$\F'$ have the same color $\chi_E$ and $\chi_N$, respectively. If
$(\chi_E,\chi_N)=(1,1)$, then $e$ behaves like the constant function
on $F'$. If $(\chi_E,\chi_N)=(2,3)$, then it behaves like the
identity, and if $(\chi_E,\chi_N)=(3,2)$, then $e$ behaves like $-$.
If $(\chi_E,\chi_N)=(2,2)$ or $(\chi_E,\chi_N)=(3,3)$, then $e$
behaves like $e_E$ or $e_N$, respectively. Finally, it is easy to
see that $(\chi_E,\chi_N)=(1,q)$ or $(\chi_E,\chi_N)=(q,1)$, where
$q\in\{2,3\}$, is impossible if $\F$ contains the two three-element
graphs with one and two edges, respectively.

\end{proof}

\begin{defn}\label{def:canonical}
    Let $\U=(U;F)$ be a graph, and let $f: U\To U$. Let $S_1,S_2$ be disjoint subsets of $U$.
    We say that $f$ is \emph{canonical on $S_1$} iff it behaves the same way on all edges and on all non-edges,
    respectively: This is to say that if $f$ collapses one edge in
    $S_1$, then it collapses all edges; if it makes an edge a
    non-edge, then it does so for all edges; etc.
    Similarly, we say that $f$ is \emph{canonical between $S_1$ and $S_2$} iff the same holds for all edges and non-edges between $S_1$ and $S_2$.
\end{defn}

We will often view $\U$ as a subgraph of the random graph, and $f$
will be injective. In this situation, $f$ is canonical on $S_1$ and
between $S_1,S_2$ iff it behaves like the identity, $-$, $e_E$, or
$e_N$ on $S_1$ and between $S_1,S_2$, respectively. Observe that
what we really proved in Proposition~\ref{prop:interpolation} is
that any $e: V\To V$ is canonical on arbitrarily large subgraphs of
the random graph.

\begin{defn}
    Let $\U=(U;F,U_1,\ldots,U_n)$ be a partitioned graph, and let $f: U\To U$.
    We say that $f$ is \emph{canonical on $\U$} iff it is canonical on all $U_i$ and between all distinct $U_i,U_j$.
    If $\U=(U;F,u_1,\ldots,u_n)$ is an $n$-constant graph, and $f: U\To
    U$, then $f$ is \emph{canonical on $\U$} iff it is canonical on the corresponding $n+2^n$-partitioned graph.
\end{defn}

\begin{defn}
    We call a countable structure \emph{$\aleph_0$-universal} iff it embeds all finite structures of the same signature.
\end{defn}

\begin{lem}[The $n$-partite graph interpolation lemma]\label{lem:n-partite-interpolation}
        Let $\U=(U;C,U_1,\ldots,U_n)$ be an $\aleph_0$-universal partitioned graph, and let $f: U\To U$.
        Then every finite partitioned graph has a copy in $\U$
        on which $f$ is canonical.
\end{lem}
\begin{proof}
    This is immediate from the $n$-partitioned graph Ramsey lemma (Lemma~\ref{lem:partitionedGraphRamseyLemma}):
    Just like in the proof of Proposition~\ref{prop:interpolation}, we color the edges and non-edges of $\U$ according to what $f$
    does to them.
\end{proof}

\begin{lem}[The $n$-constant graph interpolation lemma]\label{lem:n-constant-interpolation}
        Let $\U=(U;C,u_1,\ldots,u_n)$ be an $\aleph_0$-universal $n$-constant graph, and let $f: U\To U$.
        Then every finite $n$-constant graph has a copy in $\U$
        on which $f$ is canonical.
\end{lem}
\begin{proof}
    This is immediate from the $n$-constant graph Ramsey lemma (Lemma~\ref{lem:constantGraphRamseyLemma}).
\end{proof}

\section{Proof of the Main theorem}

We will apply Lemma~\ref{lem:n-constant-interpolation} to prove

\begin{lem}\label{lem:eN}
    Let $e: V\To V$ be so that it preserves $N$ but not $E$. Then $e$ generates $e_N$.
\end{lem}
\begin{proof}
    We prove that for every finite subset $F$ of $V$,
    $e$ produces an operation which behaves like $e_N$ on $F$. We
    first claim that there are adjacent vertices $a,b\in V$ such that $N(e(a),e(b))$.
     Since $e$ does not preserve $E$, there exist $u,v$ with
    $(u,v)\in E$ such that $(e(u),e(v))\nin E$. If $N(e(u),e(v))$, then we are done.
    If $e(u)=e(v)$, then choose $w$ such that $E(w,u)$ and $N(w,v)$. We have $(e(w),e(u))=(e(w),e(v))\in N$, so $u,w$ prove the claim.

    Now, $\U:=(V;E,a,b)$ is an $\aleph_0$-universal 2-constant graph. Therefore, by
    Lemma~\ref{lem:n-constant-interpolation}, $e$ is canonical on
    arbitrarily large substructures of $\U$. Since $e$ preserves
    $N$, it is easy to see that if $e$ is canonical on a $2$-constant
    graph which is large enough, then  $e$ must be injective. (For
    example, if $e$ is canonical on a graph which contains the three-element graph with two edges, then
    $e$ cannot collapse any edges of that graph.) Hence, $e$ is canonical
    and injective on arbitrarily large $2$-constant subgraphs of
    $\U$. Since $e$ preserves $N$, we have that for arbitrarily large substructures of
    $\U$, it behaves like the identity or like $e_N$ on and between
    the parts of these structures; in particular, it does not add any
    edges. Hence, for any finite $2$-constant graph, we can delete
    the edge between the two constants without adding any other
    edges. But that means that starting from any finite graph, we
    can delete all edges by repeating this process, choosing any
    edge we want to get rid of in each step. This proves the lemma.
\end{proof}

The following is just the dual statement.

\begin{cor}\label{lem:eE}
    Let $e: V\To V$ be so that it preserves $E$ but not $N$. Then $e$ generates $e_E$.
\end{cor}

\begin{lem}\label{lem:constant}
    Let $e: V\To V$ be so that it preserves neither $E$ nor $N$. If $e$ is not injective, then $e$ generates a constant operation.
\end{lem}
\begin{proof}
    We must show that for any finite subset $F$ of $V$, $e$ generates an operation which is constant on $F$.

    Observe that $e$ generates operations $g,h$ which collapse an edge
    and a non-edge, respectively. To see this, note that since
    $e$ is not injective, it collapses an edge or a non-edge; say
    without loss of generality it collapses an edge, so we can set
    $g:=e$. If it also collapses a non-edge, then we are done.
    Otherwise, since $e$ violates $N$, it sends some non-edge to an
    edge, which, with the help of an appropriate automorphism, can
    be collapsed by another application of $e$.

    Having this, once proceeds inductively to collapse all the vertices of
    $F$, shifting $F$ around with automorphisms accordingly and applying $g$ and
    $h$. After at most $|F|$ steps, the whole of $F$ is collapsed to a
    single vertex.
\end{proof}

The following theorem states that there exist five functions on the
random graph which are minimal in the sense that every function
which is not an automorphism of $G$ generates one of these five
functions.

\begin{theorem}\label{thm:endos-weak}
Let $\Gamma$ be first-order definable in the random graph. Then one
of the following cases applies.
\begin{enumerate}
\item $\Gamma$ has a constant endomorphism.
\item $\Gamma$ has $e_E$ as an endomorphism.
\item $\Gamma$ has $e_N$ as an endomorphism.
\item $\Gamma$ has $-$ as an automorphism.
\item $\Gamma$ has {\it sw} as an automorphism.
\item All endomorphisms of $\Gamma$ are locally generated by the automorphisms of the random graph.
\end{enumerate}
\end{theorem}

\begin{proof}

If $\Gamma$ has an endomorphism $e$ which preserves $E$ but not $N$
or $N$ but not $E$, then we can refer to Lemma~\ref{lem:eN} and
Corollary~\ref{lem:eE}. If all of its endomorphisms preserve both
$N$ and $E$, then they are all generated by the automorphisms of
$G$. We thus assume henceforth that $\Gamma$ has an endomorphism $e$
which violates both $E$ and $N$.

If $e$ is not injective, then it generates a constant operation, by
Lemma~\ref{lem:constant}. So suppose that $e$ is injective. Fix
distinct $x,y$ such that $E(x,y)$ and $N(e(x),e(y))$.

By Proposition~\ref{prop:interpolation}, $e$ is canonical on
arbitrarily large finite subgraphs of $G$. If $e$ interpolates $-$,
$e_E$, or $e_N$ modulo automorphisms, then we are done. So assume
this is not the case, i.e., there is a finite graph $\F_0$ with the
property that on all copies of $\F_0$ in $G$, $e$ does not behave
like any of these operations. Observe that $e$ then behaves like the
identity on arbitrarily large subgraphs of $G$. Moreover, this
assumption implies that if only a finite subgraph $\F$ of $G$ is
sufficiently large (i.e., if it embeds $\F_0$), and $e$ is canonical
on $\F$, then $e$ behaves like the identity on $\F$.

We now make a series of observations which rule out bad behavior of
$e$ between subsets of the random graph, and which follow from our
assumptions of the preceding paragraph; the easily verifiable
details are left to the reader.

\begin{itemize}
\item If $e$ behaves like $-$ between the parts of arbitrarily large
finite 2-partitioned subgraphs of $G$, then it generates $\sw$.

\item If $e$ behaves like $e_N$ between the parts of arbitrarily
large finite 2-partitioned subgraphs of $G$, then it generates
$e_N$.

\item If $e$ behaves like $e_E$ between the parts
 of arbitrarily large finite 2-partitioned subgraphs of $G$, then
 it generates $e_E$.
\end{itemize}

We assume therefore that for sufficiently large finite
$2$-partitioned subgraphs of $G$, if $e$ is canonical on such a
graph, then $e$ behaves like the identity on and between the parts.

Now observe that $\Q:=(V;E,x,y)$ is an $\aleph_0$-universal
2-constant graph. Let $\F=(F;D,f_1,f_2)$ be any finite 2-constant
graph. By the $n$-constant interpolation lemma
(Lemma~\ref{lem:n-constant-interpolation}), there is a copy $\F'$ of
$\F$ in $\Q$ on which $e$ is canonical. By our assumption above, if
only $\F$ is large enough, then being canonical on a proper part
$F_i'$ of the 6-partitioned graph
$\tilde{\F'}=(F';E,\{x\},\{y\},F_1',\ldots,F_4')$ corresponding to
$\F'$ means behaving like the identity thereon, and being canonical
between proper parts means behaving like the identity between these
parts. Therefore, all 2-constant graphs $\F$ have a a copy
$\F'=(F';E,x,y)$ in $\Q$ such that $e$ behaves like the identity on
and between all of the parts $F_i', F_j'$ of the corresponding
partitioned graph $\tilde{\F'}=(F';E,\{x\},\{y\},F'_1,\ldots,F'_4)$.

Of a two-constant graph $\F$, consider the reduct $\H=(F;D,f_1)$.
This reduct has a copy $\H'$ in $\Q^x=(V;E,x)$ on which $e$ is
canonical. The corresponding partitioned graph has two parts $H_1'$,
$H_2'$, and $x$ is connected to, say, all vertices in $H_1'$ and to
none in $H_2'$. Since $e$ is canonical on $\H'$, either all edges
leading to $H_1'$ are kept or deleted. Similarly with the non-edges
between $x$ and $H_2'$. If all edges are deleted and all non-edges
kept for arbitrarily large $\H$, then $e$ generates $e_N$. If all
edges are deleted and all non-edges edged for arbitrarily large
$\H$, then $e$ interpolates $\sw$ modulo automorphisms. If all edges
are kept and all non-edges edged for arbitrarily large $\H$, then
$e$ generates $e_E$. So we assume that if only $\H$ is large enough,
then all edges and non-edges are kept by $e$ on those copies of $\H$
on which $e$ is canonical.

We use the same argument with the reduct $(F;D,f_2)$ and
$\Q^y=(V;E,y)$, and arrive at the conclusion that if the
two-constant graph $\F$ is large enough, then on every copy of $\F$
in $\Q$ which $e$ is canonical on, the edges and non-edges leading
from $x$ and $y$ to the other vertices of the copy are kept.

Combining this with what we have established before, we conclude
that if only $\F$ is large enough, and $\F'$ is a copy of $\F$ in
$\Q$ which $e$ is canonical on, then $e$ behaves like the identity
on $\F'$ except between $x$ and $y$, where it deletes the edge.
Hence, for any finite $\F$ we can find a copy in $\Q$ on which $e$
behaves that way. But this implies that starting from any finite
graph $\S:=(F;D)$, we can pick any edge in $\S$, say between
vertices $f_1,f_2$, and then find a copy of $\F:=(F;D,f_1,f_2)$ in
$\Q$ such that $e$ deletes exactly that edge from the copy whithout
changing the rest. Hence, by shifting finite graphs around with
automorphisms, we can delete a single edge from an arbitrary finite
subgraph of $G$ without changing the rest of the graph. Applying
this successively, we can remove all edges from arbitrary finite
graphs, proving that $e$ generates $e_N$.
\end{proof}

Proving Theorem~\ref{thm:endos} now amounts to showing that if cases
(1),(2),(3), and (6) of Theorem~\ref{thm:endos-weak} do not apply
for a structure $\Gamma$, and hence if (4) or (5) of that theorem
hold, then its endomorphisms are generated by its automorphisms.
This will be accomplished in the three propositions to come.

\begin{proposition}\label{prop:above-minus}
Let $\Gamma$ be first-order definable in the random graph, and
suppose $\Gamma$ is preserved by $-$ but not by $e_N, e_E$, or a
constant operation. Then the endomorphisms of $\Gamma$ are locally
generated by $\{ -\} \cup \Aut(G)$, or $\Gamma$ is preserved by {\it
sw}.
\end{proposition}
\begin{proof}

Suppose the endomorphisms of $\Gamma$ are not generated by $\{ -\}
\cup \text{Aut}(G)$. Then, by Proposition~\ref{prop:loc-gen}, there
is a relation $R$ invariant under $\{ -\} \cup \text{Aut}(G)$ and an
endomorphism $e$ of $\Gamma$ which violates $R$; that is, there
exists a tuple $a:=(a_1,\ldots,a_n)\in R$ such that
$e(a)=(e(a_1),\ldots,e(a_n))\nin R$.

Since $R$ is definable in the random graph, $e$ violates either an
edge or a non-edge. Hence, as in the proof of
Theorem~\ref{thm:endos-weak}, the assumption that $e$ does not
generate $e_N$, $e_E$, or a constant operation implies that $e$ is
injective.

Let $\F=(F;D,f_1,\ldots,f_n)$ be any finite $n$-constant graph. By
the $n$-constant interpolation lemma
(Lemma~\ref{lem:n-constant-interpolation}), there is a copy $\F'$ of
$\F$ in the $\aleph_0$-universal $n$-constant graph
$\Q:=(V;E,a_1,\ldots,a_n)$ such that $e$ is canonical on this copy.

We now make a series of observations on the behavior of $e$ on and
between subsets of $V$ where it is canonical.

\begin{itemize}

\item Since by assumption, $e$ does not
interpolate $e_E$, $e_N$, or a constant operation modulo
automorphisms, it behaves like $-$ or the identity on sufficiently
large finite subgraphs of $G$ where it is canonical.

\item Suppose that for arbitrarily large finite $2$-partitioned subgraphs of
$G$, $e$ behaves like the identity on the parts and like $-$ between
the parts. Then $e$ generates $\sw$.

\item Suppose that for arbitrarily large finite $2$-partitioned subgraphs of
$G$, $e$ behaves like the identity on the parts and like $e_N$ (like
$e_E$) between the parts. Then $e$ generates $e_N$ ($e_E$).

\item Suppose that for arbitrarily large finite $2$-partitioned subgraphs of
$G$, $e$ behaves like $-$ on the parts and like the identity / $e_N$
/ $e_E$ between the parts. Then $e$ and $-$ together generate $\sw$
/ $e_E$ / $e_N$. This is because we can apply the preceding two
observations to $-e$.

\item Suppose that for arbitrarily large finite $2$-partitioned subgraphs of
$G$ which $e$ is canonical on, $e$ behaves like $-$ on one part and
like the identity on the other part. Then $e$ and $-$ together
generate $e_N$.

\end{itemize}

To see the last assertion for the case where $e$ behaves like the
identity between the parts, select an edge within one of the parts
that is mapped to a non-edge. For arbitrary finite $A \subseteq V$
we can now use the operation $e$ to get rid of one edge in the graph
induced by $A$ in $G$ and preserve all other edges, and so
eventually generate an operation that behaves like $e_N$ on $A$. For
the case where $e$ behaves like $-$ between the parts, we can apply
the same argument to $-e$. If $e$ behaves like $e_N$ between the
parts, then we can all the more delete edges. If it behaves like
$e_E$ between the parts, then $-e$ behaves like $e_N$ and we are
back in the preceding case.

Summarizing our observations, we can assume that for an arbitrary
finite $n$-constant graph $\F$ there is a copy of $\F$ in $\Q$ such
that $e$ behaves like the identity on and between all proper parts
$F_i',F_j'$ of the corresponding partitioned graph, or like $-$ on
and between all of its parts. If only the second case holds for
arbitrarily large $n$-constant graphs $\F$, then we simply proceed
our argument with $-e$ instead of $e$. We can do that since also
$-e(a)\nin R$: For otherwise, picking an automorphism $\alpha$ of
$G$ such that $\alpha (-(-x))=x$ for all $x\in V$, we would have
$\alpha(-(-e(a)))=e(a)\in R$, contrary to our choice of $a$. Thus we
assume that for arbitrary finite $n$-constant graphs $\F$ there is a
copy of $\F$ in $\Q$ such that $e$ behaves like the identity on and
between all proper parts of that copy.

As in the proof of Theorem~\ref{thm:endos-weak}, we may assume that
if a copy $\F'=(F';E,a_1,\ldots,a_n)$ of $\F$ in $\Q$ is large
enough and $e$ is canonical on $\F'$ and behaves like the identity
on and between all proper parts $F_i',F_j'$ of the corresponding
$n$-partitioned graph $\tilde{\F'}$, then it leaves the edges and
non-edges between the $a_i$ and the vertices in
$F'\setminus\{a_1,\ldots,a_n\}$ unaltered. It follows that for
arbitrary finite $n$-constant graphs $\F$ there is a copy of $\F$ in
$\Q$ such that the only edges or non-edges changed by $e$ on this
copy are those between the $a_i$.

Finally, note that since $R$ is definable in the random graph and
$e(a)\nin R$, $e$ destroys at least one edge or one non-edge on
$\{a_1,\ldots,a_n\}$. Without loss of generality, say that $a_1,
a_2$ are adjacent but their values under $e$ are not. We have shown
that for arbitrarily large $2$-constant graphs $\H$, there is a copy
of $\H$ in $(V;E,a_1,a_2)$ such that $e$ behaves like the identity
on this copy, except for the edge between $a_1$ and $a_2$, which is
destroyed. This clearly implies that $e$ generates $e_N$.
\end{proof}

\begin{proposition}\label{prop:above-switch}
Let $\Gamma$ be first-order definable in the random graph, and
suppose $\Gamma$ is preserved by {\it sw} but not by $e_N, e_E$, or
a constant operation. Then the endomorphisms of $\Gamma$ are locally
generated by $\{{\it sw}\} \cup \Aut(G)$, or $\Gamma$ is preserved
by $-$.
\end{proposition}
\begin{proof}

The proof is very similar to the proof of the preceding proposition.
 This time we know that unless the endomorphisms
are locally generated by $\{{\it sw}\} \cup \Aut(G)$, there exists
an endomorphism $e$ that violates a relation $R$ which is preserved
by $\{{\it sw}\} \cup \Aut(G)$. Fix a tuple $a$ as before.

As in the preceding proof, we may assume that $e$ is injective. If
$e$ interpolates $-$ modulo automorphisms, we are done. Suppose
therefore that if $e$ is canonical on a finite partitioned graph
large enough, then it must behave like the identity on its parts.

If $e$ behaves like $e_N$ ($e_E$) between the parts of arbitrarily
large finite $2$-partitioned subgraphs of $G$, then it generates
$e_N$ ($e_E$). Thus we may assume that it behaves like the identity
or $-$ between such parts.

Suppose that for arbitrarily large finite $3$-partitioned subgraphs
$\F=(F;E,F_1,F_2,F_3)$ of $G$ which $e$ is canonical on, $e$ behaves
like the $-$ between exactly two of the parts, say between
$F_1,F_2$, and like the identity between $F_2,F_3$ and $F_1,F_3$.
Then $e$ is easily seen to generate both $e_N$ and $e_E$. Indeed, if
we want to delete\footnote{for the purposes of the proof, we
identify ourselves with the personalized endomorphism monoid} any
edge from a finite graph, then we can view the vertices of the edge
as two parts of a $3$-partitioned graph, where the third part
contains all the other vertices. If $e$ behaves like $-$ between the
two vertices whose edge we want to delete, and like the identity on
and between the other parts, what happens is exactly that the edge
is deleted.

If for arbitrarily large finite $3$-partitioned subgraphs $\F$ of
$G$ which $e$ is canonical on, $e$ behaves like $-$ between, say,
$F_1,F_2$ and $F_1,F_3$, and like the identity between $F_2,F_3$,
then by applying a suitable switch operation $i_A$ to $e$ we are
back in the preceding case. Note here that there is an automorphism
$\alpha$ of $G$ such that $i_A(\alpha(i_A(x)))=x$ for all $x\in V$.
Therefore, $i_A(e(a))\nin R$; for otherwise, we would have
$i_A(\alpha(i_A(e(a)))=e(a)\in R$, a contradiction.

The latter argument works also if $e$ behaves like $-$ between all
three parts. Summarizing, we may assume that if $e$ is canonical on
a finite $n$-partitioned graph which is large enough, where $n\geq
3$, then it behaves like the identity on and between all of the
parts.

As for $n$-constant graphs which $e$ is canonical on, $e$ might flip
edges and non-edges between some parts and the constants. However,
this situation can easily be repaired by a single application of
$\sw$.

Finally, observe that at least one edge or one non-edge on
$a_1,\ldots,a_n$ is destroyed, and that we therefore can generate
either $e_N$ or $e_E$.
\end{proof}

\begin{proposition}\label{prop:to-all-permutations}
Let $\Gamma$ be first-order definable in the random graph, and
suppose $\Gamma$ is preserved by $\sw$ and by $-$, but not by $e_N,
e_E$, or a constant operation. Then the endomorphisms of $\Gamma$
are locally generated by $\{-,{\it sw}\} \cup \Aut(G)$, or $\Gamma$
is preserved by all permutations.
\end{proposition}
\begin{proof}
The argument goes as in the preceding two propositions; we leave the
details to the reader.
\end{proof}

Theorem~\ref{thm:endos} now is a direct consequence of
Theorem~\ref{thm:endos-weak}, and
Propositions~\ref{prop:above-minus}, \ref{prop:above-switch},
\ref{prop:to-all-permutations}: If a reduct $\Gamma$ of $G$ does not
have $e_E$, $e_N$, or a constant operation as an endomorphism, and
if its endomorphisms are not generated by the automorphisms of $G$,
then Theorem~\ref{thm:endos-weak} implies that it has either $-$ or
$\sw$ as an endomorphism. Since $\Aut(\Gamma)$ contains $\Aut(G)$,
once $\Gamma$ has $-$ or $\sw$ as an endomorphism, it also has its
inverse as an endomorphism; thus it has $-$ or $\sw$ as an
automorphism. But then by the preceding three propositions, either
$\End(\Gamma)$ is generated by $\Aut(\Gamma)$, or $\Gamma$ is
preserved by all permutations. The latter case, however, is
impossible, as this would imply that $e_E$ and $e_N$ are among its
endomorphisms, which we excluded already.

Observe also how Thomas' classification of closed permutation groups
containing $\Aut(G)$ (Theorem~\ref{thm:reducts}) follows from our
results: If a group properly contains $\Aut(G)$, then it contains
$-$ or $\sw$, by Theorem~\ref{thm:endos-weak}. If it contains $-$
but is not generated by $-$, then it contains $\sw$ by
Proposition~\ref{prop:above-minus}. Similarly, if it contains $\sw$
but is not generated by $\sw$, then it contains $-$ by
Proposition~\ref{prop:above-switch}. If it contains both $-$ and
$\sw$, but is not generated by these operations, then it must
already contain all permutations
(Proposition~\ref{prop:to-all-permutations}).

%\bibliographystyle{plain}
%\bibliography{local}

\end{document}